\documentclass[runningheads]{llncs}

\newcommand{\refsec}[1] {Section~\ref{#1}}
\newcommand{\refeq}[1] {(\ref{#1})}
\newcommand{\refeqeq}[2] {(\ref{#1})--(\ref{#2})}
\newcommand{\reffig}[1] {Figure~\ref{#1}}

\renewcommand{\vec}[1] {\boldsymbol{#1}}
\newcommand{\mat}[1] {\boldsymbol{\mathsf{#1}}}

\newcommand{\diff}{\mathrm{d}}

\newcommand{\norm}[1] {\left\|#1\right\|}

\usepackage{amsmath}
\usepackage{amssymb}
\usepackage{graphicx}
\usepackage{ucs}
\usepackage[utf8x]{inputenc}

\usepackage{subcaption}
\usepackage{tikz}
\usepackage[outline]{contour}

\usepackage[T1]{fontenc}
\usepackage{amsmath,amssymb}
\usepackage{authblk}
\usepackage{booktabs}
\usepackage{hyperref}
\hypersetup{
    colorlinks,
    hidelinks,
    citecolor=blue,
    filecolor=black,
    linkcolor=red,
    urlcolor=green,
    plainpages=false,
    pdfpagelabels=true
}

\usepackage{url}
\newcommand{\keywords}[1]{\par\addvspace\baselineskip
\noindent\keywordname\enspace\ignorespaces#1}

\newcommand{\R}{\mathbb{R}}
\newcommand{\E}{\mathbb{E}}
\newcommand{\be}{\vec{e}}
\newcommand{\br}{\vec{r}}
\newcommand{\bd}{\vec{d}}
\newcommand{\bk}{\vec{\kappa}}
\newcommand{\bm}{\vec{m}}
\newcommand{\bn}{\vec{n}}
\newcommand{\bff}{\vec{f}}
\newcommand{\bw}{\vec{\omega}}
\newcommand{\bv}{\vec{v}}
\newcommand{\bnu}{\vec{\nu}}
\newcommand{\bJ}{\vec{J}}
\newcommand{\bl}{\vec{L}}
\newcommand{\bh}{\vec{h}}
\newcommand{\bF}{\vec{F}}
\newcommand{\bp}{\vec{p}}
\newcommand{\bq}{\vec{q}}

\newcommand{\bPhi}{\vec{\Phi}}

\begin{document}

\mainmatter

\title{On the General Analytical Solution of the Kinematic Cosserat Equations}
\titlerunning{On the General Analytical Solution of the Kinematic Cosserat Equations}

\author{Dominik~L.~Michels\inst{1
}
\and Dmitry~A.~Lyakhov\inst{2} \and Vladimir~P.~Gerdt\inst{3} \and Zahid~Hossain\inst{1,4} \and Ingmar~H.~Riedel-Kruse\inst{4} \and Andreas~G.~Weber\inst{5}}
\institute{Department of Computer Science, Stanford University, 353 Serra Mall, Stanford, CA 94305, USA\\
\email{michels@cs.stanford.edu}\\
\and
Visual Computing Center, King Abdullah University of Science and Technology, Al Khawarizmi Building, Thuwal, 23955-6900, Saudi Arabia\\
\email{dmitry.lyakhov@kaust.edu.sa}\\
\and
Group of Algebraic and Quantum Computations, Joint Institute for Nuclear Research, Joliot-Curie 6, 141980 Dubna, Moscow Region, Russia\\
\email{gerdt@jinr.ru}
\and
Department of Bioengineering, Stanford University, 318 Campus Drive, Stanford, CA 94305, USA\\
\email{ingmar@stanford.edu}, \email{zhossain@stanford.edu}
\and
Institute of Computer Science II, University of Bonn, Friedrich-Ebert-Allee 144, 53113 Bonn, Germany\\
\email{weber@cs.uni-bonn.de}}

\authorrunning{D.~L.~Michels et al.}
\toctitle{Lecture Notes in Computer Science}
\tocauthor{D.~L.~Michels et al.}


\maketitle

\begin{abstract}
Based on a Lie symmetry analysis, we construct a closed form solution to the kinematic part of the (partial differential) Cosserat equations describing the mechanical behavior of elastic rods. The solution depends on two arbitrary analytical vector functions and is analytical everywhere except a certain domain of the independent variables in which one of the arbitrary vector functions satisfies 
a simple explicitly given algebraic relation. As our main theoretical result, in addition to the construction of the solution, we proof its generality. Based on this observation, a hybrid semi-analytical solver for highly viscous two-way coupled fluid-rod problems is developed which allows for the interactive high-fidelity simulations of flagellated microswimmers as a result of a substantial reduction of the numerical stiffness.
\end{abstract}

\section{Introduction}
\label{sec:1}

Studying the dynamics of nearly one-dimensional structures has various scientific and industrial applications, for example in biophysics (cf.~\cite{Elgeti:2015,Goldstein:2015} and the references therein) and visual computing (cf.~\cite{MMS:2015}) as well as in civil and mechanical engineering (cf.~\cite{Boyer:2011}), microelectronics and robotics (cf.~\cite{CaoTucker:2008}). In this regard, an appropriate description of the dynamical behavior of flexible one-dimensional structures is provided by the so-called special Cosserat theory of elastic rods (cf.~\cite{Antman:1995}, Ch.~8, and the original work \cite{Cosserat:1909}). This is a general and geometrically exact dynamical model that takes bending, extension, shear, and torsion into account as well as rod deformations under external forces and torques. In this context, the dynamics of a rod is described by a governing system of twelve first-order nonlinear partial differential equations (PDEs) with a pair of independent variables $(s,t)$ where $s$ is the arc-length and $t$ the time parameter. In this PDE system, the two kinematic vector equations ((9a)--(9b) in~\cite{Antman:1995}, Ch.~8) are parameter free and represent the compatibility conditions for four vector functions $\bk,\bw,\bnu,\bv$ in $(s,t)$. Whereas the first vector equation only contains two vector functions $\bk,\bw$, the second one contains all four vector functions $\bk,\bw,\bnu,\bv$. The remaining two vector equations in the governing system are dynamical equations of motion and include two more dependent vector variables $\hat{\bm}(s,t)$ and $\hat{\bn}(s,t)$. Moreover, these dynamical equations contain parameters (or parametric functions of $s$) to characterize the rod and to include the external forces and torques.

Because of its inherent stiffness caused by the different deformation modes of a Cosserat rod, a pure numerical treatment of the full Cosserat PDE system requires the application of specific solvers; see e.g.~\cite{LangLinnArnold:2011,MLGSW:2015}. In order to reduce the computational overhead caused by the stiffness, we analyzed the Lie symmetries of the first kinematic vector equation ((9a) in~\cite{Antman:1995}, Ch.~8) and constructed its general and (locally) analytical solution in \cite{MLGSW:2014} which depends on three arbitrary functions in $(s,t)$ and three arbitrary functions in $t$.

In this contribution we perform a computer algebra-based Lie symmetry analysis to integrate the full kinematic part of the governing Cosserat system based on our previous work in \cite{MLGSW:2014}. This allows for the construction of a general analytical solution of this part which depends on six arbitrary functions in $(s,t)$. We prove its generality and apply the obtained analytical solution in order to solve the dynamical part of the governing system. Finally, we prove its practicability by simulating the dynamics of a flagellated microswimmer. To allow for an efficient solution process of the determining equations for the infinitesimal Lie symmetry generators, we make use of the Maple package {\sc SADE} (cf.~\cite{SADE:2011}) in addition to {\sc Desolv} (cf.~\cite{CarminatiVu:2000}).

This paper is organized as follows. \refsec{sec:2} describes the governing PDE system in the special Cosserat theory of rods. In \refsec{sec:3}, we show that the functional arbitrariness in the analytical solution to the first kinematic vector equation that we constructed in \cite{MLGSW:2014} can be narrowed down to three arbitrary bivariate functions. Our main theoretical result is presented in \refsec{sec:4}, in which we construct a general analytical solution to the kinematic part of the governing equations by integrating the Lie equations for a one-parameter subgroup of the Lie symmetry group. \refsec{sec:5} illustrates the practicability of this approach by realizing a semi-analytical simulation of a flagellated microswimmer. This is based on a combination of the analytical solution of the kinematic part of the Cosserat PDE and a numerical solution of its dynamical part. Some concluding remarks are given in \refsec{sec:6} and limitations are discussed in \refsec{sec:7}.

\section{Special Cosserat Theory of Rods}
\label{sec:2}

In the context of the special Cosserat theory of rods (cf.~\cite{Antman:1995,CaoTucker:2008,Cosserat:1909,MLGSW:2014}), the motion of a rod is defined by a vector-valued function
\[
   [a,b]\times \R \ni (s,t) \mapsto \left(\br(s,t),\,\bd_1(s,t),\,\bd_2(s,t)\right)\in \E^3\,. \label{rd1d2}
\]
Here, $t$ denotes the time and $s$ is the arc-length parameter identifying a {\em material cross-section} of the rod which consists of all material points whose reference positions are on the plane perpendicular to the rod at $s$. Moreover, $\bd_1(s,t)$ and $\bd_2(s,t)$ are orthonormal vectors, and $\br(s,t)$ denotes the position of the material point on the centerline with arc-length parameter $s$ at time $t$. The Euclidean 3-space is denoted with $\E^3$. The vectors $\bd_1,\,\bd_2$, and $\bd_3:=\bd_1\times \bd_2$ are called {\em directors} and form a right-handed orthonormal moving frame. The use of the triple $(\bd_1,\,\bd_2,\,\bd_3)$ is natural for the intrinsic description of the rod deformation whereas $\br$ describes the motion of the rod relative to the fixed frame $(\be_1,\,\be_2,\,\be_3)$. This is illustrated in \reffig{Fig1}.

From the orthonormality of the directors follows the existence of so-called {\em Darboux} and {\em twist} vector functions $\bk=\sum_{k=1}^3\kappa_k\bd_k$ and $\bw=\sum_{k=1}^3\omega_k\bd_k$ determined by the kinematic relations
\begin{equation}
         \partial_s\bd_k=\bk\times \bd_k\,,\quad \partial_t\bd_k=\bw\times \bd_k\,. \label{kr}
\end{equation}
The {\em linear strain} of the rod and the {\em velocity of the material cross-section} are given by vector functions $\bnu:=\partial_s\br=\sum_{k=1}^3\nu_k\bd_k$ and $\bv:=\partial_t\br=\sum_{k=1}^3v_k\bd_k$.

\contourlength{1.6pt}
\newcommand{\renderScale}{0.12}
\newcommand{\outlineScale}{1.25*\renderScale}
\newcommand{\pxToMetric}{0.0352777778}
\newcommand{\horizontalCrop}{0.3}
\newcommand{\imageWidth}{1920}
\newcommand{\imageHeight}{1080}

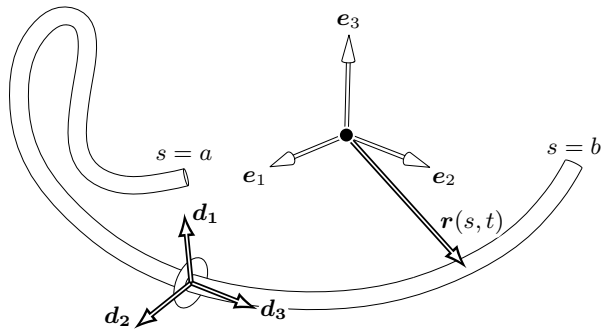
\begin{figure}
\centering
\begin{tikzpicture}
\begin{scope}[scale=10*\pxToMetric*\renderScale]
\coordinate (d1) at (609.500mm, 393.654mm);
\coordinate (d2) at (451.588mm, 30.955mm);
\coordinate (d3) at (839.662mm, 92.118mm);
\coordinate (o1) at (631.054mm, 174.816mm);
\coordinate (ax) at (887.632mm, 543.492mm);
\coordinate (ay) at (1369.466mm, 541.844mm);
\coordinate (az) at (1126.066mm, 941.984mm);
\coordinate (o2) at (1118.542mm, 639.549mm);
\coordinate (r) at (1495.859mm, 221.154mm);
\coordinate (s0) at (608.083mm, 533.215mm);
\coordinate (sl) at (1832.694mm, 552.899mm);
\end{scope}

\tikzset{
  double -latex/.style args={#1 colored by #2 and #3}{
    -latex,line width=#1,#2,
    postaction={draw,-latex,#3,line width=(#1)/3,shorten <=(#1)/4,shorten >=4.5*(#1)/3},
  },
  double round cap-latex/.style args={#1 colored by #2 and #3}{
    round cap-latex,line width=#1,#2,
    postaction={draw,round cap-latex,#3,line width=(#1)/3,shorten <=(#1)/4,shorten >=4.5*(#1)/3},
  },
  double round cap-stealth/.style args={#1 colored by #2 and #3}{
    round cap-stealth,line width=#1,#2,
    postaction={round cap-stealth,draw,,#3,line width=(#1)/3,shorten <=(#1)/3,shorten >=2*(#1)/3},
  },
  double -stealth/.style args={#1 colored by #2 and #3}{
    -stealth,line width=#1,#2,
    postaction={-stealth,draw,,#3,line width=(#1)/3,shorten <=(#1)/3,shorten >=2*(#1)/3},
  },
  double -triangle 45/.style args={#1 colored by #2 and #3}{
    -triangle 45,line width=#1,#2,
    postaction={-triangle 45,draw,,#3,line width=(#1)/3,shorten <=(#1)/3,shorten >=3.5*(#1)/2},
  },
}

\path[use as bounding box] (\horizontalCrop, 0) rectangle (\renderScale*\pxToMetric*\imageWidth - \horizontalCrop, \renderScale*\pxToMetric*\imageHeight);

\node[scale=\outlineScale, anchor=south west, inner sep=0pt] at (0, 0) {\includegraphics{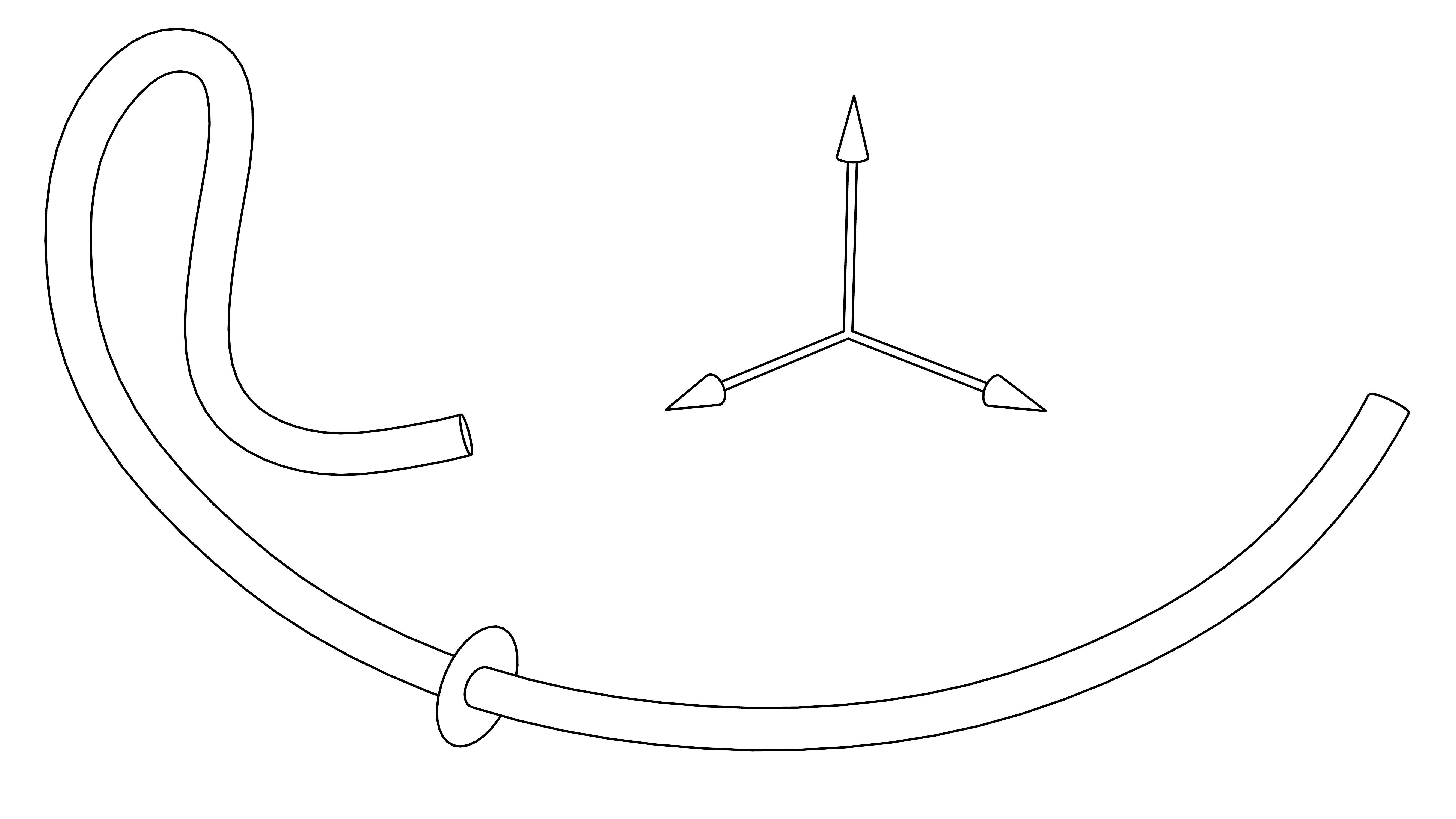}};

\draw[double -latex=2pt colored by black and white] (o2) -- (r) node[black, pos=0.65, right, inner sep=5pt, xshift=0pt, yshift=0pt] {\contour{white}{$\boldsymbol{r}(s, t)$}};

\draw[white, thick, fill=black] (o2) circle (0.1);

\draw[double -latex=2pt colored by black and white] (o1) -- (d1) node[black, anchor=west, inner sep=2pt, xshift=0pt, yshift=0pt] {\contour{white}{$\boldsymbol{d_1}$}};
\draw[double -latex=2pt colored by black and white] (o1) -- (d2) node[black, anchor=south east, inner sep=0pt, xshift=0pt, yshift=0pt] {\contour{white}{$\boldsymbol{d_2}$}};
\draw[double -latex=2pt colored by black and white] (o1) -- (d3) node[black, anchor=west, inner sep=0pt, xshift=0pt, yshift=0pt] {\contour{white}{$\boldsymbol{d_3}$}};

\node [black, anchor=north east, inner sep=2pt, xshift=0pt, yshift=0pt] at (ax) {\contour{white}{$\boldsymbol{e}_1$}};
\node [black, anchor=north west, inner sep=2pt, xshift=0pt, yshift=0pt] at (ay) {\contour{white}{$\boldsymbol{e}_2$}};
\node [black, anchor=south, inner sep=4pt, xshift=0pt, yshift=0pt] at (az) {\contour{white}{$\boldsymbol{e}_3$}};

\node [black, anchor=south, inner sep=3pt, xshift=0pt, yshift=0pt] at (s0) {\contour{white}{$s = a$}};
\node [black, anchor=south, inner sep=3pt, xshift=0pt, yshift=0pt] at (sl) {\contour{white}{$s = b$}};

\end{tikzpicture}
\caption{The vector set $\{\vec{d}_1,\vec{d}_2,\vec{d}_3\}$ forms a right-handed orthonormal basis. The directors $\vec{d}_1$ and $\vec{d}_2$ span the local material cross-section, whereas $\vec{d}_3$ is perpendicular to the cross-section. Note that in the presence of shear deformations  $\vec{d}_3$ is unequal to the tangent $\partial_s\vec{r}$ of the centerline of the rod.}
\label{Fig1}
\end{figure}

\noindent
The components of the {\em strain variables} $\bk$ and $\bnu$ describe the deformation of the rod: the flexure with respect to the two major axes of the cross-section $(\kappa_1,\kappa_2)$, torsion $(\kappa_3)$, shear $(\nu_1,\nu_2)$, and extension $(\nu_3)$.

The triples
\begin{equation}
\bk=(\kappa_1,\kappa_2,\kappa_3)\,,\ \bw=(\omega_1,\omega_2,\omega_3)\,,\ \bnu=(\nu_1,\nu_2,\nu_3)\,,\ \bv=(v_1,v_2,v_3)          \label{kv}
\end{equation}
are functions in $(s,t)$, that satisfy the {\em compatibility conditions}
\begin{equation}
\partial_t\partial_s\bd_k=\partial_s\partial_t\bd_k\,,\quad \partial_t\partial_s\br=\partial_s\partial_t\br\,. \label{cc}
\end{equation}
The substitution of~\eqref{kr} into the left equation in~\eqref{cc} leads to
\[
  \tilde{\bk}_t=\tilde{\bw}_s-\bk\times \bw\quad \text{with}\quad \tilde{\bk}_t=\partial_t\sum_{k=1}^3\kappa_k\bd_k\,,\ \tilde{\bw}_s=\partial_s\sum_{k=1}^3\omega_k\bd_k\,.
\]
On the other hand one obtains, $\tilde{\bk}_t=\bk_t+\bw\times\bk$
 and $\tilde{\bw}_s=\bw_s+\bk\times\bw$ with $\bk_t=(\partial_t\kappa_1,\partial_t\kappa_2,\partial_t\kappa_3)$ and $\bw_s=(\partial_s\omega_1,\partial_s\omega_2,\partial_s\omega_3)$, and therefore
\begin{equation}
  \bk_t=\bw_s - \bw\times \bk\,. \label{ke1}
\end{equation}
Similarly, the second compatibility condition in~\eqref{cc} is equivalent to
\begin{equation}
   \bnu_t=\bv_s+\bk\times\bv-\bw\times \bnu \label{ke2}
\end{equation}
with $\bnu_t=(\partial_t\nu_1,\partial_t\nu_2,\partial_t\nu_3)$ and $\bv_s=(\partial_sv_1,\partial_sv_2,\partial_sv_3)$.

The first-order PDE system~\refeq{ke1}--\refeq{ke2} with independent variables $(s,t)$ and dependent variables \refeq{kv} forms the kinematic part of the governing Cosserat equations ((9a)--(9b) in~\cite{Antman:1995}, Ch.~8). The construction of its general solution is the main theoretical result of this paper.

The remaining part of the governing equations in the special Cosserat theory consists of two vector equations resulting from Newton's laws of motion. For a rod density $\rho(s)$ and cross-section $A(s)$, these equations are given by
\begin{eqnarray*}
\begin{array}{l}
\rho(s) A(s)\partial_t\bv=\partial_s\bn(s,t)+\bF(s,t)\,,\\[0.2cm]
\partial_t\bh(s,t)=\partial_s\bm(s,t)+\bnu(s,t)\times \bn(s,t)+\bl(s,t)\,,
\end{array}
\label{nl}
\end{eqnarray*}
where $\bm(s,t)=\sum_{k=1}^3m_k(s,t)\,\bd_k(s,t)$ are the {\em contact torques},  $\bn(s,t)=\sum_{k=1}^3n_k(s,t)\,\bd_k(s,t)$ are the {\em contact forces}, $\bh(s,t)=\sum_{k=1}^3h_k(s,t)\,\bd_k(s,t)$ are the {\em angular momenta}, and $\bF(s,t)$ and $\bl(s,t)$ are the  {\em external forces} and {\em torque densities}.

The contact torques $\bm(s,t)$ and contact forces $\bn(s,t)$ corresponding to the {\em internal stresses}, are related to the extension and shear strains $\bnu(s,t)$ as well as to the flexure and torsion strains $\bk(s,t)$ by the
{\em constitutive relations}
\begin{equation} \label{cons_rel}
 \bm(s,t)=\hat{\bm}\left(\bk(s,t),\bnu(s,t),s\right)\,,\quad \bn(s,t)=\hat{\bn}\left(\bk(s,t),\bnu(s,t),s\right)\,.
\end{equation}
Under certain reasonable assumptions (cf.~\cite{Antman:1995,CaoTucker:2008,MLGSW:2014}) on the structure of the right-hand sides of \eqref{cons_rel}, together with the kinematic relations~\eqref{ke1} and~\eqref{ke2}, it yields to the governing equations (cf.~\cite{Antman:1995}, Ch.~8, (9.5a)--(9.5d))
\begin{eqnarray}
\begin{array}{l}
\bk_t=\bw_s - \bw\times \bk\,,\\[0.1cm]
\bnu_t=\bv_s+\bk\times\bv-\bw\times \bnu\,,\\[0.15cm]
\rho\bJ\cdot \bw_t=\hat{\bm}_s+\bk\times\hat{\bm}+\bnu\times \hat{\bn}-\bw\times(\rho\bJ\cdot\bw)+\bl\,,\\[0.15cm]
\rho A\bv_t=\bn_s+\bk\times \hat{\bn}-\bw\times (\rho A\bv)+\bF\,,
\end{array}
\label{sct}
\end{eqnarray}
in which $\bJ$ is the inertia tensor of the cross-section per unit length. The dynamical part of~\eqref{sct} contains parameters characterizing the rod under consideration of $\rho,A,\bJ$ and the external force and torque densities $\bF$ and $\bl$, whereas the kinematic part is parameter free.

\section{Analytical Form of the Darboux and Twist Functions}
\label{sec:3}

In \cite{MLGSW:2014}, we constructed a general solution to \eqref{ke1} that is the first equation in the PDE system~\eqref{sct}. In so doing, we proved that the constructed solution is (locally) analytical and provides the structure of the twist vector function $\bw$ and the Darboux vector function $\bk$:
\begin{equation}\label{gen_sol}
\begin{split}
\vec{\omega} =& \vec{f} - \frac{\sin(p)}{p}\,\vec{p}\times \vec{f} + \frac{1-\cos(p)}{p^2}\left(\vec{p}\,(\vec{p}\cdot \vec{f})-p^2\,\vec{f}\right)+\\
  & + {\vec{p}_t}+\frac{p-\sin(p)}{p^3}\,\left(\vec{p}\,\left(\vec{p}\cdot {\vec{p}_t}\right)-p^2\,{ \vec{p}_t}\right)-\frac{1-\cos(p)}{p^2}\,\vec{p}\times {\vec{p}_t}\,,\\
\vec{\kappa} = & {\vec{p}_s}
  +\frac{p-\sin(p)}{p^3}\,\left(\vec{p}\,\left(\vec{p}\,{\cdot \vec{p}_s}\right)-p^2\,{\vec{p}_s}\right)-\frac{1-\cos(p)}{p^2}\,\vec{p}\times {\vec{p}_s}\,,
\end{split}
\end{equation}
where $\bff:=\left(f_1(t),f_2(t),f_3(t)\right)$ and $\bp:=\left(p_1(s,t),p_2(s,t),p_3(s,t)\right)$ are arbitrary vector-valued analytical functions, and $p^2:=p_1^2+p_2^2+p_3^2$.

It turns out that the functional arbitrariness of $\bff$ and $\bp$ is superfluous, and that \eqref{gen_sol} with $\bff(t)=0$ is still a general solution to \eqref{ke1}. This fact is formulated in the following proposition.
\begin{proposition}
\label{pro:1}
The vector functions $\vec{\omega}$ and $\vec{\kappa}$ expressed by
\begin{subequations}
\begin{align}
\vec{\omega} =& {\vec{p}_t}+\frac{p-\sin(p)}{p^3}\,\left(\vec{p}\,\left(\vec{p}\cdot {\vec{p}_t}\right)-p^2\,{\vec{p}_t}\right)-\frac{1-\cos(p)}{p^2}\,\vec{p}\times {\vec{p}_t}\,,\label{omega}\\
\vec{\kappa} = & {\vec{p}_s}
  +\frac{p-\sin(p)}{p^3}\,\left(\vec{p}\,\left(\vec{p}\cdot{\vec{p}_s}\right)-p^2\,{\vec{p}_s}\right)-\frac{1-\cos(p)}{p^2}\,\vec{p}\times {\vec{p}_s}\,\label{kappa}
\end{align}
\end{subequations}
with an arbitrary analytical vector function $\bp(s,t)$, are a general analytical solution to \eqref{ke1}.
\end{proposition}
\begin{proof}
Let $(s_0,t_0)$ be a fixed point. The right-hand sides of~\eqref{omega} and~\eqref{kappa} satisfy \eqref{ke1} for an arbitrary vector function $\bp(s,t)$ analytical in $(s_0,t_0)$. It is an obvious consequence of the fact that~\eqref{gen_sol} is a solution to~\eqref{ke1} for arbitrary $\bff(t)$ analytical in $t_0$.

Also, the equalities~\eqref{omega} and~\eqref{kappa} can be transformed into each other with
\begin{equation}\label{swap}
      \bw(s,t) \Leftrightarrow \bk(s,t)\ \ \text{and}\ \ \partial_s \Leftrightarrow \partial_t
\end{equation}
reflecting the invariance of~\eqref{ke1} under~\eqref{swap}. The equalities~\eqref{omega} and~\eqref{kappa} are linear with respect to the partial derivatives $\bp_t$ and $\bp_s$, and their corresponding Jacobians. The determinants of the Jacobian matrices $J_{\bw}(\partial_tp_1,\partial_tp_2,\partial_tp_3)$ and $J_{\bk}(\partial_sp_1,\partial_sp_2,
 \partial_sp_3)$ coincide because of the symmetry~\eqref{swap} and read\footnote{The equalities in~\eqref{Jacobian} are easily verifiable with Maple (cf.~\cite{MLGSW:2014}), Sec.~3.5.}
\begin{equation}
 J(\bp):=\det\left(J_{\bw}\right)=\det\left(J_{\bk}\right)=2\,\frac{\cos(p)-1}{p^2} \label{Jacobian}\,.
\end{equation}

Let $\bw(s,t)$ and $\bk(s,t)$ be two arbitrary vector functions analytical in $(s_0,t_0)$. We have to show that there is a vector function $\bp(s,t)$ analytical in $(s_0,t_0)$ satisfying~\eqref{omega} and~\eqref{kappa}. For that, chose real constants $a,b,c$ such that
\[
   \frac{\cos(\sqrt{a^2+b^2+c^2})-1}{a^2+b^2+c^2}\neq 0
\]
and set $\bp_0:=\{a,b,c\}$. Then \eqref{omega} and~\eqref{kappa} are solvable with respect to the partial derivatives of $\bp_t$ and $\bp_s$ in a vicinity of $(s_0,t_0)$, and we obtain the first-order PDE system of the form
\begin{equation}\label{psys}
 \bp_t=\bPhi(\bw,\bp)\,,\quad \bp_s=\bPhi(\bk,\bp)\,,
\end{equation}
where the vector function $\bPhi$ is linear in its first argument and analytical in $\bp$ at $\bp_0$.

Also, the system~\eqref{psys} inherits the symmetry under the swap~\eqref{swap} and is {\em passive} and {\em orthonomic} in the sense of the Riquer-Janet theory~(cf.~\cite{Seiler:2010} and the references therein), since its vector-valued {\em passivity (integrability) condition}
$$\partial_s\bPhi(\bw,\bp)-\partial_t\bPhi(\bk,\bp)=\vec{0}$$ holds due to symmetry. Therefore, by Riquier's existence theorems~\cite{Thomas:1929} that generalize the Cauchy-Kovalevskaya theorem, there is a {\em unique} solution $\bp(s,t)$ of~\eqref{psys} analytical in $(s_0,t_0)$ and satisfying $\bp(s_0,t_0)=\bp_0$.\\\phantom{}\hfill $\Box$
\end{proof}

\section{General Solution to the Kinematic Equation System}
\label{sec:4}

In this section, we determine a general analytical form of the vector functions $\bnu(s,t)$ and $\bv(s,t)$ in~\eqref{kv} describing the linear strain of a Cosserat rod and its velocity. These functions satisfy the second kinematic equation~\eqref{ke2} of the governing PDE system~\eqref{sct} under the condition that the Darboux and the twist functions, $\bk(s,t)$ and $\bw(s,t)$, occurring in the last equation, are given by \eqref{omega} and \eqref{kappa} which contain the arbitrary analytical vector function $\bp(s,t)$.

Similarly, as we carried it out in~\cite{MLGSW:2014} for the integration of~\eqref{ke1}, we analyze Lie symmetries  (cf.~\cite{Olivery:2010} and the references therein) and consider the {\em infinitesimal generator}
\begin{equation}
{\cal{X}}:=\xi^{1}{\partial_s}+\xi^{2}{\partial_t}+\sum_{i=1}^3\left(\theta^{i}{\partial_{ \omega_i}}+\vartheta^{i}{\partial_{\kappa_i}}+\phi^{i}{\partial_{\nu_i}}+\varphi^{i}{\partial_{v_i}}\right) \label{generator}
\end{equation}
of a Lie group of point symmetry transformations for~\eqref{ke1}--\eqref{ke2}. The coefficients $\xi^{1},\xi^{2},\theta^{i},\vartheta^{j},\phi^{m},\varphi^{n}$ with $i,j,m,n\in\{1,2,3\}$ in~\eqref{generator} are functions of the independent and dependent variables.

The {\em infinitesimal criterion of invariance} of~\eqref{ke1}--\eqref{ke2} reads
\begin{equation}
{\cal{X}}^{(pr)}\vec{h}_1={\cal{X}}^{(pr)}\vec{h}_2=0\quad \text{whenever}\quad \vec{h}_1=\vec{h}_2=0\,, \label{criterion}
\end{equation}
where
\begin{equation}\label{lhs}
\vec{h}_1:=\bk_t-\bw_s + \bw\times \bk\,,\quad \vec{h}_2:=\bnu_t-\bv_s-\bk\times\bv+\bw\times \bnu\,.
\end{equation}
In addition to those in \eqref{generator}, the {\em prolonged} infinitesimal symmetry generator $X^{(pr)}$ contains extra terms caused by the presence of the first-order partial derivatives in \eqref{ke1} and~\eqref{ke2}.

The invariance conditions~\eqref{criterion} lead to an overdetermined system of linear PDEs in the coefficients of the infinitesimal generator~\eqref{generator}. This {\em determining} system can be easily computed by any modern computer algebra software (cf.~\cite{Butcher:2003}). We make use of the Maple package {\sc Desolv} (cf.~\cite{CarminatiVu:2000}) which computes the determining system and outputs 138 PDEs.

Since the completion of the determining systems to involution is the most universal algorithmic tool of their analysis (cf.~\cite{Butcher:2003,Hereman:1996}), we apply the Maple package {\sc Janet} (cf.~\cite{BCGPR:03}) first and compute a Janet involutive basis (cf.~\cite{Robertz:2014}) of 263 elements for the determining system, which took about 80 minutes of computation time on standard hardware.\footnote{The computation time has been measured on a machine with an Intel(R) Xeon E5 with 3.5 GHz and 32 GB DDR-RAM.} Then we detected the functional arbitrariness in the general solution of the determining system by means of the differential Hilbert polynomial
\begin{equation}\label{HP}
   4s^2+18s+21=8\,\binom{s+2}{s}+6\,\binom{s+1}{s}+7
\end{equation}
computable by the corresponding routine of the Maple package {\sc DifferentialThomas} (cf.~\cite{BGLHR:2012}). It shows that the general solution depends on eight arbitrary functions of $(s,t)$.
However, in contrast to the determining system for~\eqref{ke1} which is quickly and effectively solvable (cf.~\cite{MLGSW:2014}) by the routine {\em pdesolv} built in the package {\sc Desolv}, the solution found by this routine to the involutive determining system for~\eqref{ke1}--\eqref{ke2} needs around one hour of computation time and has a form which is unsatisfactory for our purposes, since the solution contains nonlocal (integral) dependencies on arbitrary functions. On the other hand, the use of {\sc SADE} (cf.~\cite{SADE:2011}) leads to a satisfying result. Unlike {\sc Desolv}, {\sc SADE} uses some heuristics to solve simpler equations first in order to simplify the remaining system. In so doing, {\sc SADE} extends the determining systems with certain integrability conditions for a partial completion to involution. In our case the routine {\em liesymmetries} of {\sc SADE} receives components of the vectors in~\eqref{lhs} and outputs the set of nine distinct solutions in just a few seconds. The output solution set includes eight arbitrary functions in $(s,t)$ which is in agreement with~\eqref{HP}. Each solution represents an infinitesimal symmetry generator~\eqref{generator}.

Among the generators, there are three that include an arbitrary vector function, which we denoted by $\bq(s,t)=\left(q_1(s,t),q_2(s,t),q_3(s,t)\right)$, with vanishing coefficients $\theta_i,\vartheta_i$, $i\in\{1,2,3\}$. The sum of these generators is given by
\begin{eqnarray}\label{subgen}
&{\cal{X}}_0:=&\left(-\partial_sq_1+q_2\kappa_3-q_3\kappa_2\right)\partial_{\nu_1}+
\left(-\partial_sq_2+q_3\kappa_1-
    q_1\kappa_2\right)\partial_{\nu_2}+\nonumber\\
&&\left(-\partial_sq_3+q_1\kappa_3-q_3\kappa_1\right)\partial_{\nu_3}+
\left(-\partial_tq_1+q_2\omega_3-q_3\omega_2\right)\partial_{v_1}+\\
&&\left(-\partial_tq_2+q_3\omega_1-q_1\omega_2\right)\partial_{v_2}+
\left(-\partial_tq_3+q_1\omega_3-q_3\omega_1\right)\partial_{v_3}\,.\nonumber
\end{eqnarray}
It generates a one-parameter Lie symmetry group of point transformations (depending on the arbitrary vector function $\bq(s,t)$) of the vector functions $\bnu(s,t)$ and $\bv(s,t)$ preserving the equality~\eqref{ke2} for fixed $\bk(s,t)$ and $\bw(s,t)$.

In accordance to Lie's first fundamental theorem (cf.~\cite{Olivery:2010}), the symmetry transformations
\[
   \bnu \mapsto \bnu'(a)\,,\quad \bv \mapsto \bv'(a)\,\quad \text{with group parameter}\quad a\in \R\,,
\]
generated by~\eqref{subgen}, are solutions to the following differential (Lie) equations whose vector form reads
\begin{equation}\label{Lie_eqs}
\mathsf{d}_a\bnu'=\bq\times \bk-\bq_s\,,\quad \mathsf{d}_a\bv'=\bq\times \bw-\bq_t\,,\quad \bnu'(0)=\bnu, \quad \bv'(0)=\bv\,.
\end{equation}
The equations~\eqref{Lie_eqs} can easily be integrated, and without a loss of generality the group parameter can be absorbed into the arbitrary function $q$. This gives the following solution\footnote{It is easy to check with Maple that the right-hand sides of~\eqref{solke2} satisfy~\eqref{ke2} for arbitrary $\bq(s,t)$ if one takes \eqref{omega} and~\eqref{kappa} into account.} to~\eqref{ke2}:
\begin{equation}\label{solke2}
   \bnu=\bq\times\bk-\bq_s\,,\quad \bv=\bq\times\bw-\bq_t\,.
\end{equation}

\begin{proposition}
The vector functions $\bw(s,t)$, $\bk(s,t)$, $\bnu(s,t)$, and $\bv(s,t)$ expressed by \eqref{omega}--\eqref{kappa} and~\eqref{solke2} with two arbitrary analytical functions $\bp(s,t)$ and $\bq(s,t)$ form a general analytical solution to~\eqref{ke1}--\eqref{ke2}.
\end{proposition}

\begin{proof}
The fact that~\eqref{omega} and~\eqref{kappa} form a general analytical solution to~\eqref{ke1} was verified in Proposition \ref{pro:1}.

We have to show that, given analytical vector functions $\bnu(s,t)$ and $\bv(s,t)$ satisfying~\eqref{ke2} with analytical $\bw(s,t)$ and $\bk(s,t)$ satisfying~\eqref{ke1}, there exists an analytical vector function $\bq(s,t)$ satisfying~\eqref{solke2}. Consider the last equalities as a system of first-order PDEs with independent variables $(s,t)$ and a dependent vector variable $\bq$. According to the argumentation in the proof of Proposition \ref{pro:1}, this leads to the fact, that the equations in \eqref{solke2} are invariant under the transformations
\begin{equation*}
   \bnu(s,t) \Leftrightarrow \bv(s,t)\,,\quad \bw(s,t) \Leftrightarrow \bk(s,t)\,,\quad \partial_s \Leftrightarrow \partial_t\,.
\end{equation*}
This symmetry implies the satisfiability of the integrability condition
\[
\partial_t(\bq\times\bk-\bnu)-\partial_s(\bq\times \bw-\bv)=0
\]
without any further constraints. Therefore, the system~\eqref{solke2} is passive (involutive), and by Riquier's existence theorem, there is a solution $\bq$ to \eqref{solke2} analytical in a point of analyticity of $\bw,\,\bk,\,\bnu,\,\bv$.\\\phantom{}\hfill$\Box$
\end{proof}

\section{Simulation of Two-Way Coupled Fluid-Rod Problems}
\label{sec:5}

To demonstrate the practical use of the analytical solution to the kinematic Cosserat equations, we combine it with the numerical solution of its dynamical part. The resulting analytical solutions \eqref{omega}--\eqref{kappa} and~\eqref{solke2} for the kinematic part of \refeq{sct} contain two parameterization functions $\bp(s,t)$ and $\bq(s,t)$, which can be determined by the numerical integration of the dynamical part of \refeq{sct}. The substitution of the resulting analytical solutions \eqref{omega}--\eqref{kappa} and~\eqref{solke2} into the latter two (dynamical) equations of \refeq{sct}, the replacement of the spatial derivatives with central differences, and the replacement of the temporal derivatives according to the numerical scheme of a forward Euler integrator, leads to an explicit expression.\footnote{We do not explicitly write out the resulting equations here for brevity. A construction of a hybrid semi-analytical, semi-numerical solver is also described in our recent contribution \cite{MLGSW:2015}.} Iterating over this recurrence equation allows for the simulation of the dynamics of a rod.

In order to embed this into a scenario close to reality, we consider a flagellated microswimmer. In particular, we simulate the dynamics of a swimming sperm cell, which is of interest in the context of simulations in biology and biophysics. Since such a highly viscous fluid scenario takes place in the low Reynolds number domain, the advection and pressure parts of the Navier-Stokes equations (cf.~\cite{Granger:1995}) can be ignored, such that the resulting so-called {\em steady Stokes equations} become linear and can be solved analytically. Therefore, numerical errors do not significantly influence the fluid simulation part for which reason this scenario is appropriate for evaluating the practicability of the analytical solution to the kinematic Cosserat equations. The {\em steady Stokes equations} are given by
\begin{eqnarray}
\mu\Delta\vec{u} = \nabla p-\vec{F}\,,\label{eq:Stokes1}\\
\nabla\cdot\vec{u}=0\,,\label{eq:Stokes2}
\end{eqnarray}
in which $\mu$ denotes the fluid viscosity, $p$ the pressure, $\vec{u}$ the velocity, and $\vec{F}$ the force. Similar to the fundamental work in \cite{Cortez:2008} we use a regularization in order to develop a suitable integration of \refeqeq{eq:Stokes1}{eq:Stokes2}. For that, we assume$$\vec{F}(\vec{x})=\vec{f}_0\,\phi_\epsilon(\vec{x}-\vec{x}_0),$$in which $\phi_\epsilon$ is a smooth and radially symmetric function with $\int\phi_\epsilon(\vec{x})\,\diff{\vec{x}}=1$, is spread over a small ball centered at the point $\vec{x}_0$.

Let $G_\epsilon$ be the corresponding Green's function, i.e., the solution of $\Delta G_\epsilon(\vec{x})=\phi_\epsilon(\vec{x})$ and let $B_\epsilon$ be the solution of $\Delta B_\epsilon(\vec{x})=G_\epsilon(\vec{x})$, both in the infinite space bounded for small $\epsilon$. Smooth approximations of $G_\epsilon$ and $B_\epsilon$ are given by $G(\vec{x})=-1/(4\pi\norm{\vec{x}})$ for $\norm{\vec{x}}>0$ and $B(\vec{x})=-\norm{\vec{x}}/(8\pi)$, the solution of the biharmonic equation $\Delta^2B(\vec{x})=\delta(\vec{x})$.

\noindent
The pressure $p$ satisfies $\Delta p=\nabla\cdot\vec{F}$, which can be shown by applying the divergence operator on \refeqeq{eq:Stokes1}{eq:Stokes2}, and is therefore given by $p=\vec{f}_0\cdot\nabla G_\epsilon$. Using this, we can rewrite \refeq{eq:Stokes1} as $$\mu\Delta\vec{u}=(\vec{f}_0\cdot\nabla)\nabla G_\epsilon-\vec{f}_0\phi_\epsilon$$ with its solution $$\mu\vec{u}(\vec{x})=(\vec{f}_0\cdot\nabla)\nabla B_\epsilon(\vec{x}-\vec{x}_0)-\vec{f}_0G_\epsilon(\vec{x}-\vec{x}_0)\,,$$ the so-called {\em regularized Stokeslet}.

For multiple forces $\vec{f}_1,\dots,\vec{f}_N$ centered at points $\vec{x}_1,\dots,\vec{x}_N$, the pressure $p$ and the velocity $\vec{u}$ can be computed by superposition. Because $G_\epsilon$ and $B_\epsilon$ are radially symmetric, we can additionally use $\nabla B_\epsilon(\vec{x})=B_\epsilon'\vec{x}/\norm{\vec{x}}$ and obtain\footnote{Since at this point, the functions $\phi_\epsilon$, $G_\epsilon$, and $B_\epsilon$ only depend on the norm of their arguments, we change the notation according to this.}
\begin{eqnarray}
p(\vec{x}) &=& \sum_{k=1}^N(\vec{f}_k\cdot(\vec{x}-\vec{x}_k))\frac{G_\epsilon'(\norm{\vec{x}-\vec{x}_k})}{\norm{\vec{x}-\vec{x}_k}}\,,\label{eq:StokesSolution1}\\
\vec{u}(\vec{x}) &=& \frac{1}{\mu}\sum_{k=1}^N\left[\vec{f}_k\left(\frac{B_\epsilon'(\norm{\vec{x}-\vec{x}_k})}{\norm{\vec{x}-\vec{x}_k}}-G_\epsilon(\norm{\vec{x}-\vec{x}_k})\right)\right.\label{eq:StokesSolution2}\\
&& +\left.(\vec{f}_k\cdot(\vec{x}-\vec{x}_k))(\vec{x}-\vec{x}_k)\frac{\norm{\vec{x}-\vec{x}_k}B_\epsilon''(\norm{\vec{x}-\vec{x}_k})-B_\epsilon'(\norm{\vec{x}-\vec{x}_k})}{\norm{\vec{x}-\vec{x}_k}^3}\right]\,.\nonumber
\end{eqnarray}
The flow given by \refeq{eq:StokesSolution2} satisfies the incompressibility constraint \refeq{eq:Stokes2}. Because of $$\Delta G_\epsilon(\norm{\vec{x}-\vec{x}_k})=\frac{1}{\norm{\vec{x}-\vec{x}_k}}(\norm{\vec{x}-\vec{x}_k}G_\epsilon'(\norm{\vec{x}-\vec{x}_k}))'=\phi_\epsilon(\norm{\vec{x}-\vec{x}_k})\,,$$ the integration of $$G_\epsilon'(\norm{\vec{x}-\vec{x}_k})=\frac{1}{\norm{\vec{x}-\vec{x}_k}}\int_0^{\norm{\vec{x}-\vec{x}_k}}s\phi_\epsilon(s)\,\diff{s}$$ leads to $G_\epsilon$. Similarly, $$\frac{1}{\norm{\vec{x}-\vec{x}_k}}(\norm{\vec{x}-\vec{x}_k}B_\epsilon'(\norm{\vec{x}-\vec{x}_k}))'=G_\epsilon(\norm{\vec{x}-\vec{x}_k})$$ leads to the expression $$B_\epsilon'(\norm{\vec{x}-\vec{x}_k})=\frac{1}{\norm{\vec{x}-\vec{x}_k}}\int_0^{\norm{\vec{x}-\vec{x}_k}}sG_\epsilon(s)\,\diff{s}$$ to determine $B_\epsilon$. We make use of the specific function $$\phi_\epsilon(\norm{\vec{x}})=\frac{15\epsilon^4}{8\pi(\norm{\vec{x}}^2+\epsilon^2)^{7/2}}\,,$$ which is smooth and radially symmetric.

Up to now, this regularized Stokeslet \refeqeq{eq:StokesSolution1}{eq:StokesSolution2} allows for the computation of the velocities for given forces. Similarly, we can tread the application of a torque by deriving an analogous {\em regularized Rodlet}; see e.g.~\cite{Ainley:2008}. In the inverse case, the velocity expressions can be rewritten in the form of the equations $$\vec{u}(\vec{x}_i)=\sum_{j=1}^NM_{ij}(\vec{x}_1,\dots,\vec{x}_N)\vec{f}_j$$ for $i\in\{1,\dots,N\}$ which can be transformed into an equation system $\vec{\mathsf{U}}=\mat{M}\,\vec{\mathsf{F}}$ with a $(3N\times3N)$-matrix $\mat{M}:=(M_{ij})_{i,j\in\{1,\dots,N\}}$. Since in general $\mat{M}$ is not regular, an iterative solver have to be applied.

A flagellated microswimmer can be set up by a rod representing the centerline of the flagellum; see \cite{Elgeti:2015}. Additionally, a constant torque perpendicular to the flagellum's base is applied to emulate the rotation of the motor. From forces and torque the velocity field is determined. Repeating this procedure to update the system state iteratively introduces a temporal domain and allows for the dynamical simulation of flagellated microswimmers; see Figures~\ref{fig:Bacteria} and \ref{fig:SpermCell}. Compared to a purely numerical handling of the two-way coupled fluid-rod system, the step size can be increased by four to five orders of magnitude, which leads to an acceleration of four orders of magnitude. This allows for real-time simulations of flagellated microswimmers on a standard desktop computer.\footnote{The simulations illustrated in Figures~\ref{fig:Bacteria} and \ref{fig:SpermCell} can be carried out in real-time on a machine with an Intel(R) Xeon E5 with 3.5 GHz and 32 GB DDR-RAM.}

\begin{figure}
\centering
\includegraphics[width=1.0\textwidth]{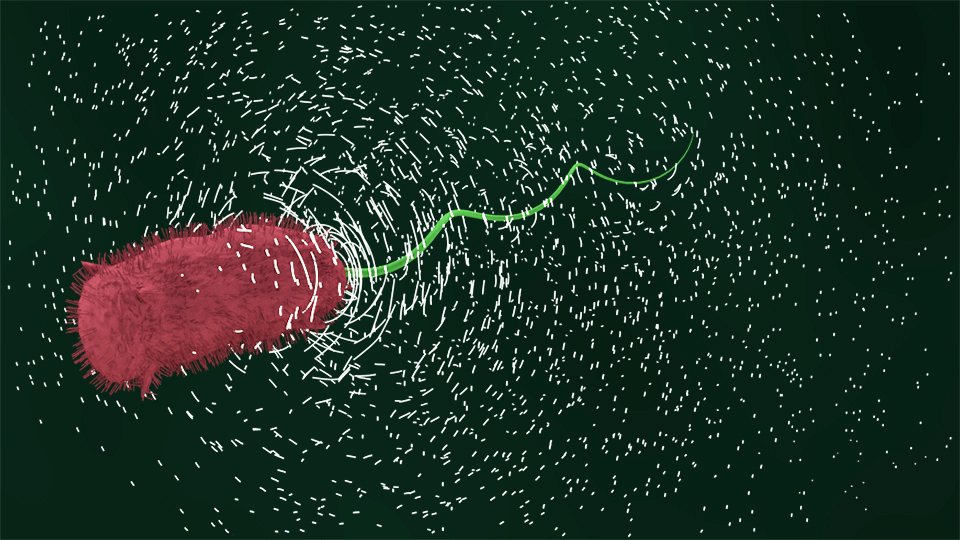}\\
\caption{Simulation of a monotrichous bacteria swimming in a viscous fluid. The rotation of the motor located at the back side of the bacteria's head causes the characteristic motion of the flagellum leading to a movement of the bacteria.}
\label{fig:Bacteria}
\end{figure}

\begin{figure}
\centering
\includegraphics[width=1.0\textwidth]{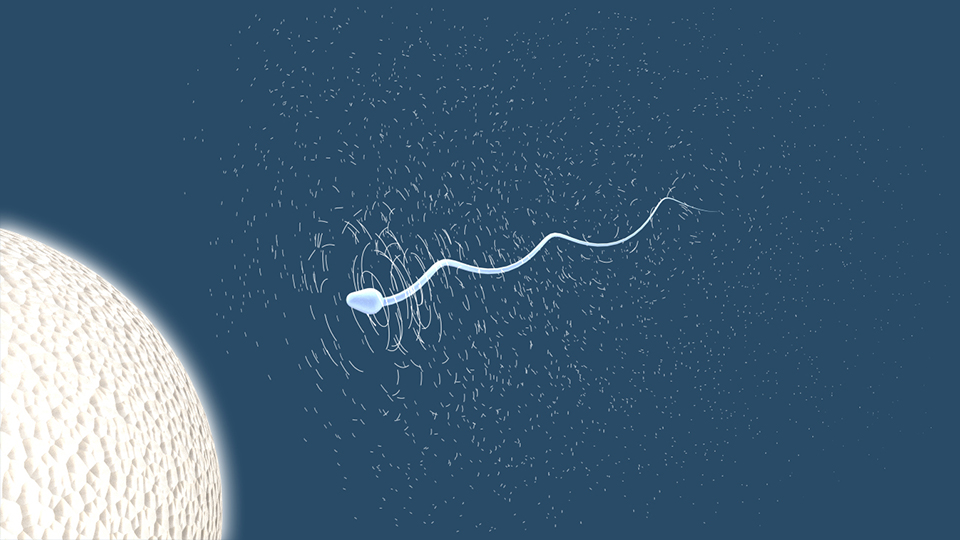}
\caption{Simulation of a sperm cell swimming into the direction of an egg. The concentration gradient induced by the egg is linearly coupled with the control of the motor. In contrast to the bacteria in \reffig{fig:Bacteria}, the flagellum of a sperm cell does not have its motor at its base as simulated here. Instead several motors are distributed along the flagellum (cf.~\cite{Riedel-Kruse:2007}), for which reason this simulation is not fully biologically accurate, but still illustrates the capabilities of the presented approach.}
\label{fig:SpermCell}
\end{figure}

\section{Conclusison}
\label{sec:6}

We constructed a closed form solution to the kinematic equations~\eqref{ke1}--\eqref{ke2} of the governing Cosserat PDE system~\eqref{sct} and proved its generality. The kinematic equations are parameter free whereas the dynamical Cosserat PDEs contain a number of parameters and parametric functions characterizing the rod under consideration of external forces and torques. The solution we found depends on two arbitrary analytical vector functions and is analytical everywhere except at the values of the independent variables $(s,t)$ for which the right-hand side of \eqref{Jacobian} vanishes. Therefore, the hardness of the numerical integration of the Cosserat system, in particular its stiffness, is substantially reduced by using the exact solution to the kinematic equations.

The application of the analytical solution prevents from numerical instabilities and allows for highly accurate and efficient simulations. This was demonstrated for the two-way coupled fluid-rod scenario of flagellated microswimmers, which could efficiently be simulated with an acceleration of four orders of magnitude compared to a purely numerical handling. This clearly shows the usefulness of the constructed analytical solution of the kinematic equations.

\section{Limitations}
\label{sec:7}

Because of the presence of parameters in the dynamical part of the Cosserat PDEs, the construction of a general closed form solution to this part is hopeless. Even if one specifies all parameters and considers the parametric functions as numerical constants, the exact integration of the dynamical equations is hardly possible. We analyzed Lie symmetries of the kinematic equations extended with one of the dynamical vector equations including all specifications of all parameters and without parametric functions. While the determining equations can be generated in a reasonable time, their completion to involution seems to be practically impossible.

\section*{Acknowledgements}

This work has been partially supported by the Max Planck Center for Visual Computing and Communication funded by the Federal Ministry of Education and Research of the Federal Republic of Germany (FKZ-01IMC01/FKZ-01IM10001), the Russian Foundation for Basic Research (16-01-00080), and a BioX Stanford Interdisciplinary Graduate Fellowship. The reviewers' valuable comments that improved the manuscript are gratefully acknowledged.


\end{document}